\documentclass[journal,twoside,web]{ieeecolor}
\usepackage{generic}
\usepackage{cite}
\usepackage{amsmath,amssymb,amsfonts}
\usepackage{algorithmic}
\usepackage{graphicx}
\usepackage{algorithm,algorithmic}
\usepackage{hyperref}
\hypersetup{hidelinks=true}
\usepackage{textcomp}
\def\BibTeX{{\rm B\kern-.05em{\sc i\kern-.025em b}\kern-.08em
    T\kern-.1667em\lower.7ex\hbox{E}\kern-.125emX}}
\usepackage{amsmath} 
\usepackage{amssymb,bm}  
\usepackage{cite}
\usepackage{graphicx} 
\usepackage{epstopdf}
\newtheorem{theorem}{Theorem}
\newtheorem{lemma}{Lemma}
\newtheorem{assumption}{Assumption}

\newtheorem{remark}{Remark}

\title{\LARGE \bf
    
    Extremum seeking with exponential convergence via high-order Lie bracket approximations}

\author{Victoria Grushkovskaya$^{1,3}$ and Sameh A. Eisa$^{2}$
\thanks{$^{1}$Victoria Grushkovskaya is with the Department of Mathematics,
        University of Klagenfurt, 9020 Klagenfurt am W\"orthersee, Austria
        {\tt\small viktoriia.grushkovska@aau.at}}%
\thanks{$^{2}$Sameh A. Eisa is with the Department of Aerospace Engineering and Engineering Mechanics, University of Cincinnati, Ohio, United states {\tt\small eisash@ucmail.uc.edu}}
\thanks{$^{3}$Institute of Applied Mathematics \& Mechanics, National Academy of Sciences of Ukraine}
}

\begin{document}
\title{Extremum seeking with exponential convergence via high-order Lie bracket approximations}
\author{Victoria Grushkovskaya, \IEEEmembership{Member, IEEE} and Sameh A. Eisa, \IEEEmembership{Member, IEEE}
\thanks{}
\thanks{Victoria Grushkovskaya is with the Department of Mathematics at the University of Klagenfurt, 9020 Klagenfurt am W\"orthersee, Austria, on leave from the Institute of Applied Mathematics \& Mechanics, NAS of Ukraine (email: viktoriia.grushkovska@aau.at)}
\thanks{Sameh A. Eisa is with the Department of Aerospace Engineering and Engineering Mechanics at the University of Cincinnati, Cincinnati, OH 45221 (email: eisash@ucmail.uc.edu).}
}

\maketitle

\begin{abstract}
This paper focuses on the further development of the Lie bracket approximation approach for \textcolor{black}{optimization and control via} extremum seeking systems. Classical results in this area provide 
algorithms with exponential convergence rates for quadratic-like cost functions, and polynomial decay rates for cost functions of higher degrees. This paper proposes a novel design that ensures the motion of the extremum seeking system along directions associated with higher-order Lie brackets, thereby achieving exponential convergence for cost functions that are \textcolor{black}{``flat-bottomed", i.e.,} polynomial-like but of degree greater than two \textcolor{black}{and unlike literature assumptions, we do not require Hessian information or strictly non zero Hessian at the minimum}. Numerical simulations are presented to demonstrate the effectiveness of the proposed designs and their exponential convergence on fourth-, sixth-, and even eighth-degree cost functions. \textcolor{black}{We include a comparison that shows our design outperforming a Newton-based method}.
\end{abstract}


\section{Introduction }
\textcolor{black}{Consider} the following control-affine system:
\begin{equation}\label{eqn:controlAffineIntro}
    \dot{x}=f_0(x)+ \sum\limits_{i=1}^{m} \frac{1}{\varepsilon^{p_i}} f_i(x) u_i\left(k_i\frac{t}{\varepsilon}\right),\\
\end{equation}
where $x =(x_1,...,x_n)^\top\in \mathbb{R}^n$ is the state  vector, \textcolor{black}{ $x(t_0)=x(0)\in\mathbb R^n$}, $p_i\in (0,1)$, $0<\varepsilon\ll1$, $f_0$ is the drift (uncontrolled) vector field of the system, $f_i$ are the control vector fields, $u_i$ are the control inputs, $m  \in \mathbb N$ is the number of control inputs, and $k_i$ are positive rational numbers. 
The geometric nature of systems of the form \eqref{eqn:controlAffineIntro} allows the application of geometric control methods and techniques \cite{bullo04,agrachev2013control} that are based on Lie-bracket-based approximations and analysis \textcolor{black}{to introduce  optimization and control laws via what is known as model-free extremum seeking systems} \cite{DurrAuto,Sch14,GZE18,Labar19,PE23}. A Lie bracket  between the vector fields $f_i\text{ and }f_j$ is defined as:
\begin{equation}
[f_i,f_j]:=\frac{\partial f_j}{\partial x}f_i-\frac{\partial f_i}{\partial x}f_j.
\end{equation}

Extremum seeking systems \cite{Ar03} are model-free, real-time dynamic optimization and control techniques aimed at stabilizing a dynamical system around the extremum point of an objective function for which only measurements, \textcolor{black}{and} not \textcolor{black}{necessarily} the analytical expression, are available; see \cite{scheinker2024100} for a comprehensive review. Classical ES methods (e.g., \cite{Kr00,yilmaz2025exponential,guay2015time}) rely on traditional averaging techniques \cite{khalil2002nonlinear,Maggia2020higherOrderAvg} for  analysis and design. On the other hand, control-affine ES methods \cite{DurrAuto,Sch14,GZE18,Labar19,PE23} (the focus of this paper) employ Lie-bracket approximations for  analysis and design. For a particular case of  \eqref{eqn:controlAffineIntro} with $p_i=0.5,k_i=1\text{ for all }i$, a first-order Lie bracket system (LBS) approximation was given in \cite{DurrAuto}, which \textcolor{black}{enabled the design of their ES law.}
In \cite{GZE18}, a generalized formulation and first-order LBS approach was presented for a class of control-affine ES systems that: (i) unifies previous works such as \cite{DurrAuto,Sch14}, and (ii) provides a new ES law with enhanced stability guarantees. It is important to emphasize that ES methods based on first-order LBS approximation (e.g., \cite{DurrAuto,Sch14,GZE18}) are gradient-based ES approaches; that is, the ES system is designed such that its LBS/average behaves as a gradient-flow \textcolor{black}{(i.e., gradient-decent-like)}. 

The paper  \cite{Labar19} further generalized  the conditions on $p_i$ and $k_i$ compared to \cite{DurrAuto}, and introduced a second-order LBS, which approximates \eqref{eqn:controlAffineIntro}. This enabled the introduction of a Newton-based ES approach, since  the inclusion of second-order Lie brackets provided access to second-order derivative information (i.e., the Hessian), hence realizing a Newton-based flow. Recently, the authors in \cite{PE23} proposed a generalized approach for generating higher-order LBSs to approximate \eqref{eqn:controlAffineIntro}. They also showed that LBS approximations can be viewed as averaging terms themselves, which guarantees the closeness of trajectories between LBSs and the original system \eqref{eqn:controlAffineIntro} given small enough $\varepsilon$. Moreover, it was observed that ES designs based on third-order LBS possess a faster convergence rate  compared even to Newton-based ES. The authors in \cite{PE23} claimed (without proof) that the observed faster convergence  arises from the inclusion of third-order Lie brackets, which provide access to higher-order derivatives beyond Hessian (i.e., third-order derivative information).

\textbf{Motivation \& contributions.} Inspired by the concept of higher-order Lie bracket averaging \cite{PE23}, this paper further explores the application of these techniques to extremum seeking problems, with a particular focus on achieving faster convergence rates. In particular, many extremum seeking algorithms based on first-order Lie bracket approximations exhibit exponential convergence when the cost function  $J(x)$ behaves locally like a quadratic function near the extremum point $x^*$  (i.e., $J(x)\sim \|x-x^*\|^{2}$). However, if the cost function behaves like a higher-degree polynomial near the extremum, i.e., $ J(x) \sim \|x - x^*\|^m  $ with $ m > 2$ \textcolor{black}{(flat-bottomed)}, such algorithms exhibit only a polynomial decay rate~\cite{GZE18}.  \textcolor{black}{In fact, $m>2$ violates assumptions needed for exponential convergence in classic/standard ES methods in general (see \cite[Assumption 1]{yilmaz2025exponential}) due to zero Hessian at the minimum. Even Newton-based methods struggle with flat-bottomed objective functions due to vanishing and extremely small Hessian (curvature) at the minimum, making its inversion and estimation very challenging numerically; also, Newton-based methods do not posses exponential convergence for high-degree objective functions.} It is observed in \cite{PE23} that the use of higher-order Lie brackets to design the ES law reached a faster convergence rate compared to the Newton-based ES law for $J(x)\sim \|x-x^*\|^{4}$. Nevertheless, no conclusion or proof was provided regarding the nature of the faster convergence resulting from higher-order Lie brackets. 
 
 In this paper, we investigate a class of extremum seeking problems involving the unconstrained minimization of a cost function $J(x)$.
The function $J$ can be unknown in terms of an explicit analytic expression, but it can be evaluated (measured) at any point. We focus on designing an ES system of the form
$
 \dot x = u(t,J(x)),
$
such that the system's trajectories  tend exponentially to an extremum point of the function $J$. To develop our approach, in this paper we assume that the cost function  behaves locally like an $m$-th order near the minimizer $x^*\in\mathbb R^n$, i.e., $J(x)\sim \|x-x^*\|^{m}$, with some $m\ge 2$. 

The main contribution of this paper is a novel extremum seeking design framework that leverages the excitation of higher-order Lie brackets to steer the system along directions corresponding to higher-order derivatives of the cost function. This approach helps to increase the convergence rate, in particular,  ensuring the exponential convergence even for cost functions that are not quadratic in nature. Furthermore, we generalize the result of~\cite{GZE18}, which described a family of vector fields whose first-order Lie bracket equals the gradient of the cost function. In this paper, we extend this idea by deriving a formula that generates vector fields such that the corresponding $\ell$-th order Lie bracket equals the $\ell$-th derivative of the cost function. 
This paper extends the conference version~\cite{grushkovskaya2025extremum} in several aspects. First, we introduce an explicit formula for exciting iterated Lie brackets of the form $[f_1,[f_1,...,[f_1,f_2]]]$ of arbitrary length in two-input control-affine systems without drift. This allows us to  achieve exponential convergence in extremum seeking systems with higher-order polynomial cost functions. Moreover, this result is relevant not only for extremum seeking  tasks, but also for various stabilization and motion planning problems for two-input nonholonomic systems. Arising from the findings of~\cite{Gau14,GZ24_CDC}, our formula strengthens prior results by providing \emph{explicit} ES laws that generate the desired Lie brackets with prescribed coefficients. Second, we discuss the decay rate of the solutions and illustrate it via the simulations. Third, we extend the introduced ES design to multi-variable cost functions and propose sufficient conditions for exponential convergences in such cases. Finally, 
additional simulations included in this extended paper demonstrate the effectiveness of the proposed ES approach for fourth-, sixth-, and eighth-order polynomial cost functions, as well as in multi-variable cases. \textcolor{black}{We also compare the convergence of the proposed ES system vs. a Newton-based ES method to demonstrate the advantage of the new ES approach handeling ``flat-bottomed" objective functions as opposed to Newton-based methods.} 
Besides that, we include a block diagram of the generalized design and expand several proofs to make the paper more accessible and self-contained.

\section{Preliminaries}

\subsection{Notations} 

$\mathbb R^+=[0,\infty)$  -- the set of all non-negative real numbers;\\
 $B_\delta(x^*)$,$\overline{B_\delta(x^*)}$ -- $\delta$-neighborhood of $x^*{\in} \mathbb R^n$ and its closure;\\
$C^\ell(D;\mathbb R)$ -- the set of all functions $h: D \to \mathbb{R}$ that are $\ell$-times continuously differentiable on $D$;\\
 for  $h{\in} C^1(\mathbb R^n;\mathbb R)$, $\xi{\in}\mathbb R^n$, we define 
 $\nabla h(\xi):=\frac{\partial h(x)}{\partial x}^T\Big|_{x=\xi}$ to be a column vector;
\\
for $h\in C^\ell(\mathbb R;\mathbb R)$, $\ell\in\mathbb N$, we define  $h^{(\ell)}(x):=\frac{d^\ell h(x)}{dx^\ell}$;\\
for $f:\mathbb R\to\mathbb R$,  $f(z)=O(z)$ as $z\to 0$ means that there is a $c>0$ such that $|f(z)|\le c|z|$ in some neighborhood of $0$;\\
for  $f,g:\in C^1(\mathbb R^n;\mathbb R^n)$, $x^*\in\mathbb R^n$, the Lie derivative at $x^*$ is defined as
 $L_gf(x)=\left.\frac{\partial f(x)}{\partial x}g(x)\right|_{x=x^*}$, 
 and 
 $[f,g](x^*)=L_fg(x^*)-L_gf(x^*)$ is the first order Lie bracket. 
 \\
 to define higher-order Lie brackets, we introduce $\ell+1$-dimensional multi-index $\mathcal I_\ell=(i_1,... ,i_{\ell+1})$; then\\
 $f_{\mathcal I_\ell}(x)= \left[f_{i_1},[f_{i_2},...,[f_{i_{\ell}},f_{i_{\ell+1}}],...]\right](x)$ --  the left-iterated Lie bracket of length $(\ell+1)$, or $\ell$-th-order Lie bracket;
for $\mathcal I_0{\in}\{1,...,m\}$, $f_{I_0}$  means a corresponding  vector field;\\
for $\ell\in\mathbb N$, $f_1,f_2\in C^\ell$, we will also use the notation
${\rm ad}_{f_1^\ell}f_2=\big[\underbrace{f_1,[f_1,...,[f_1}_{\ell\text{ times}},f_2]...]\big]$;\\

\subsection{Lie brackets approximations} 
Consider the control-affine system \eqref{eqn:controlAffineIntro}
assuming $f_i{:\mathbb R^n\to\mathbb R^n}$ be continuously differentiable (up to any order) vector fields, and $u_i$ be continuous in $t$ and $T$-periodic functions with zero average, i.e., $\int\limits_0^Tu_i(\tau)d\tau=0$. LBS approximations of \eqref{eqn:controlAffineIntro} up to a third-order are \cite{PE23}:
\begin{equation}\label{eqn:LBS_main}
    \dot {\bar x}{=}f_0(\bar x)+ \sum_{i = 1}^r L_i(\bar x),
    \end{equation}
with 
$ L_1 = \sum\limits_{j_1 = 1}^m \sum\limits_{j_2=j_1+1}^m \sigma_{j_1 j_2}[f_{j_1}, f_{j_2}]$,
\begin{align}
        L_2 &= \sum_{j_1 = 1}^m \sum_{j_2=j_1+1}^m \sum_{j_3=1}^m  \sigma_{j_1 j_2 j_3} [f_{j_3},[f_{j_1},f_{j_2}]],\nonumber\\
        \begin{split}\nonumber
        L_3 &= \sum_{j_1 = 1}^m \sum_{j_2=j_1+1}^m \sum_{j_3=1}^m \sum_{j_4=j_3+1}^m \beta_{1_{j_1 j_2 j_3 j_4}} \Big[[f_{j_1},f_{j_2}],[f_{j_3},f_{j_4}]\Big]\\
        &+ \sum_{j_1 = 1}^m \sum_{j_2=j_1+1}^m \sum_{j_3=1}^m \sum_{j_4=1}^m \beta_{2_{j_1 j_2 j_3 j_4}} [[[f_{j_1},f_{j_2}],f_{j_3}],f_{j_4}],
    \end{split}
 \end{align}
where $\sigma_{j_1 j_2}$, $\sigma_{j_1 j_2 j_3}$, $\beta_{1_{j_1 j_2 j_3 j_4}}$ and $\beta_{2_{j_1 j_2 j_3 j_4}}$ are coefficients resulting from the iterated integrals of the dither input signals (formulas are provided in \cite[Section 4]{PE23}). 
Truncating \eqref{eqn:LBS_main} at $r=1$ provides first-order LBS (e.g., \cite{DurrAuto,GZE18}). Similarly, truncating \eqref{eqn:LBS_main} at $r=2$ and $r=3$ provides second- and third-order LBS, respectively.
   The stability properties of systems~\eqref{eqn:controlAffineIntro} and \eqref{eqn:LBS_main} truncated at a finite $r$ are related as follows. 
  \begin{lemma}[\cite{DurrAuto,GZE18,PE23}]~\label{dthm}
\textit{If a compact set $ S\subset\mathbb R^n$ is locally (globally) uniformly asymptotically stable for~\eqref{eqn:LBS_main} then it is locally (semi-globally) practically
uniformly asymptotically stable for~\eqref{eqn:controlAffineIntro}.}
\end{lemma}	
\begin{remark} 
\textcolor{black}{Lemma 1 is standard in ES literature and it means that for any positive time $T>0$ trajectories of the ES system \eqref{eqn:controlAffineIntro} stays close to corresponding trajectories of the LBS \eqref{eqn:LBS_main} truncated at some $r$; hence, if the LBS converges asymptotically to an equilibrium point, the original ES system  ~\eqref{eqn:controlAffineIntro} stays close to that equilibrium while the error between \eqref{eqn:controlAffineIntro} and \eqref{eqn:LBS_main} can be made small as needed as $\varepsilon \to 0$. The reader can refer \cite{DurrAuto,GZE18,PE23} for formal definitions due to space limitation.}
\end{remark}
Now, let us recap the approach from \cite{GZE18} given its relevance to the contributions of this paper. For a special class of \eqref{eqn:controlAffineIntro} when $f_0=0$, $p_i=0.5$, $k_i=1$ for all $i$, the result of Lemma~\ref{dthm} with $r=1$ (i.e. first-order LBS), can be exploited for solving the extremum seeking problem in the following way~\cite{DurrAuto,GZE18}: let us consider a class of ES systems of the form
\begin{equation}\label{int}
\dot x=\frac{1}{\sqrt\varepsilon}\Big(F_1(J(x))u_{1}^\varepsilon( t)+F_{2}(J(x)) u_{2}^\varepsilon( t)\Big),
\end{equation}
where  $F_1\circ J,F_2\circ J\in C^2(D)$ satisfy the relation \textcolor{black}{ $[F_1,F_2](z)=1$} for all $z\in\mathbb R$ (which implies $[F_1\circ J,F_2\circ J](x)=\nabla J(x)$), $u_1^\varepsilon(t)=2\sqrt{{\pi}{\varepsilon^{-1}}}\cos\big({2\pi  t}{\varepsilon^{-1}}\big)$, $u_2^\varepsilon(t)=2\sqrt{{\pi}{\varepsilon^{-1}}}\sin\big({2\pi  t}{\varepsilon^{-1}}\big)$, and assume that the cost function $J$ satisfies the following properties in a domain $D\subseteq\mathbb R^n$~\textcolor{black}{\cite{GZE18}}:
\begin{assumption}
The function $J\in C^m(D,\mathbb R^n)$, and
\begin{itemize}
    \item [A1.1) ] there exists an $x^*{\in} D$  such that   $\nabla J(x)=0$   if and only if $x=x^*,$
    and $J(x^*)=J^*{\in}\mathbb R$, $J(x)>J(x^*)$ for all $x\in D{\setminus}\{x^*\}$.
    \item[A1.2) ] There exist constants $\alpha_1,\alpha_2,\beta_1,\beta_2,\mu$, and $m\ge 1$, such that, for all {$x\in D$},
     $$ \begin{aligned}
     \alpha_1\|x{-}x^*\|^{m} \le &J(x)-J^* \le \alpha_2\|x{-}x^* \|^{m},\\
      \beta_1 (J(x)-J^*)^{1{-}\frac{1}{m}}\le  & \|\nabla J(x)\|\le\beta_2 (J(x)-J^*)^{1{-}\frac{1}{m}},\\
    \left\|\frac{\partial^2 J(x)}{\partial x^2}\right\|{\le}&\mu (J(x)-J^*)^{1{-}\frac{2}{m}}.
    \end{aligned}
    $$
\end{itemize}
\begin{remark}
    \textcolor{black}{Assumption A1.1) states that the cost function $J$ possesses an isolated local minimum at $x^*$,where it attains the value  $J^*$. We generalized the smoothness condition to $C^m(D,\mathbb R^n)$ as opposed to $C^2(D,\mathbb R^n)$ in \cite{GZE18} given the higher-order LBS approach used in this paper. Assumption A1.2) reflects the requirement that $J$ is convex and exhibits a local behavior similar to that of a power function. For example, if $m=2$, the local behavior of $J$ is quadratic (similar to classic and Lie-bracket-based ES methods \cite{Kr00,yilmaz2025exponential,DurrAuto,GZE18,yilmaz2025unbiased}) and if $m>2$, $J$ behaves locally like a higher-order polynomial (i.e., non-quadratic, flat-bottomed). That is, a Taylor expansion of $J$ near $x^*$ is dominated by a polynomial of power $m$ term. We note that A1.2) is relevant to the literature results in \cite{GZE18} and some of the following re-called results. However, we will refine, simplify and customize A1.2) later in Assumption 2 for our new results for higher-order.}   
\end{remark}
\end{assumption}
For first-order Lie bracket ES, It can be shown that the point $x=x^*$ is practically asymptotically stable for~\eqref{int}, with the convergence rate dependent on the parameter $m$ in A1.2). Namely, the following result follows from~\cite{GZE18}:
\begin{lemma}\label{lem_decay}
  \textit{If the cost function $J\in C^2(\mathbb R^n;\mathbb R)$ satisfies  Assumption~1 in a domain $D\subset \mathbb R^n$, then $x^*$ is
 practically exponentially stable for system~\eqref{int} if $m=2$, and $x^*$ is  practically asymptotically stable for system~\eqref{int} if $m>2$.
Namely, for any $\delta$ such that $\overline{B_\delta(x^*)}\subset D$,  any $\bar\gamma\in(0,\alpha\beta_1)$, and $\rho\in(0,\delta)$, there exists an $\bar\varepsilon>0$ such that, for any $\varepsilon\in(0,\bar\varepsilon]$,
$\gamma\in(0,\bar\gamma]$, the solutions of system~\eqref{int} with $x^0{\in} B_{\delta}(x^*)$ exhibit the following decay rate:
\begin{itemize}
    \item if $m=2$, then
    $$\|x(t)-x^*\|\le \lambda_{m}\|x^0-x^*\|e^{-\gamma t}+\rho;$$
    \item if $m>2$, then
    $$
    \|x(t)-x^*\|\le \Big(\lambda_{m}\|x^0-x^*\|^{2-m}+\tilde  \lambda_m t\Big)^{-1/(m-2)}+\rho.
    $$
\end{itemize}
Here $ \gamma_m,\tilde  \gamma_m$ are positive constants; $ \gamma_m$  can be made arbitrarily close to 1 by choosing a sufficiently small $\varepsilon$. }
\end{lemma}
More technical details on the above decay rate estimates can be found in~\cite{GZE18}.

 For the sake of clarity, we have assumed here $x\in\mathbb R$, however, the above result can be easily extended to the case $x\in\mathbb R^n$~(see, e.g.,~\cite{DurrAuto,GZE18}).

\subsection{Main idea of the proposed approach}

As follows from Lemma~\ref{lem_decay}, the Lie bracket approximation approach provides a constructive solution to the extremum seeking problem which ensures the exponential convergence of the trajectories of system~\eqref{int} to the optimal point in the case of a quadratic-like cost function, i.e., for $m=2$ in Assumption A2). However, for cases $m>2$, the above algorithm ensures only a polynomial decay rate. This can also be observed by analyzing the first-order Lie bracket system associated with system~\eqref{int}:
$$
\dot{\bar x}=-\nabla J(\bar x).
$$
Note that the above first-order LBS is a complete average asymptote for \eqref{int}, meaning that higher-order LBS approximations will be redundant \cite{PE23}. Now, if for example $J=\dfrac{1}{2}(x-x^*)^{2}$, $x,x^*\in\mathbb R$, then the Lie bracket system is linear, $\dot{\bar x}=-(\bar x-x^*) $, and thus $x^*$ is its exponentially stable equilibrium point. In case $J=\frac{1}{4}(x-x^*)^{4}$, the Lie bracket system takes the form $\dot{\bar x}=-(\bar x-x^*)^3$, and its solutions are well-known to exhibit the polynomial decay rate $O(t^{-1/2})$ as $t\to\infty$~\cite{GZ13,G16}.  Assume now that we can associate the properties of system~\eqref{eqn:controlAffineIntro} with a system which gives access to the third-order derivative of $J$, namely, with the system 
$\dot{\tilde x}=-J^{(3)}(\tilde x)=-6(\tilde x-x^*)$. Then the latter system turns out to be linear again, which, under certain assumptions, may imply the practical exponential stability of the extremum seeking system. A natural way of accessing higher order derivatives of the cost function is to excite the Lie brackets of corresponding order~\cite{Labar19,PE23}. For example, it is easy to see that
$
[1,[1,[1,J]]](x)=J^{(3)}(x).
$
The main idea of this paper is to construct an extremum seeking system in the form
\begin{equation}
    \label{int_gen}
    \dot x=\sum_{k=1}^{n_u} g_k(J(x))u_k^\varepsilon(t),
\end{equation}
so that, under a special choice of control vector fields $g_k$, $k=1,..., n_u\in\mathbb N$, dither signals $u_k^\varepsilon\in L^\infty_{[0,\varepsilon]}$, and a parameter $\varepsilon>0$,
 its trajectories approximate the trajectories of a system with high order Lie brackets:
 \begin{equation}
     \label{Lie_gen}
     \dot{\bar x}=\sum_{\ell=1}^N\sum_{\mathcal I_\ell\in\mathcal S_\ell}c_{\mathcal I_\ell}g_{\mathcal I_\ell}(\bar x),
 \end{equation}
 where  $N\in\mathbb N$,  $\mathcal S_\ell\subset\{1,...,n_u\}^{\ell+1}$ denotes the sets of of multi-indices of the Lie brackets required for solving the extremum seeking problem, $g_{\mathcal I_\ell}$ are the corresponding Lie brackets, and $c_{\mathcal I_\ell}$ are constant parameters. For example, for system~\eqref{eqn:LBS_main} with $r=1$, $f_0=0$, $f_i=g_i$, $i\in\{1,...,n\}$, we mean $N=1$,   $\mathcal S_1=\{(i,j):1\le i<j\le n\}$, $\mathcal I_2:=(i,j)\in \mathcal S_1$, $c_{\mathcal I_2}=\sigma_{ij}$.
 
One of the main tools exploited in this paper is the Chen--Fliess series expansion: under certain regularity assumptions on the control vector fields of system~\eqref{int_gen}, the solutions of system~\eqref{int_gen} with $x(0)=x^0$ can be represented in the following form~\cite{La95,GZ23}:
\begin{equation}
    \label{series_gen}
    \begin{aligned}
     &  x(t)=x^0+\sum_{\ell=1}^N\sum_{k_1,..., k_{\ell}=1}^{n_u}L_{g_{k_{\ell}}}... L_{g_{k_2}}f_{g_{k_1}}(x^0) \\
       &\cdot\int\limits_0^t\int\limits_0^{s_1}... \int\limits_0^{s_{\ell}}u_{k_1}^\varepsilon(s_1)\cdot ... u_{k_\ell}^\varepsilon(s_\ell) ds_\ell... ds_1 + R(t),
    \end{aligned}
\end{equation}
with the remainder
$$
\begin{aligned}
      R(t)=&\hspace{-1em}\sum_{k_1,..., k_{N+1}=1}^{n_u}\int\limits_0^t\int\limits_0^{s_1}... \int\limits_0^{s_{N}}L_{g_{k_{N+1}}}... L_{g_{k_2}}g_{k_1}(x(s_{N+1})) \\
       &\cdot u_{k_1}^\varepsilon(s_1)\cdot ...  \cdot u_{k_{N+1}}^\varepsilon(s_{N+1}) ds_{N+1}... ds_1.
    \end{aligned}
$$
 
\section{Design of two-input extremum seeking system}

\subsection{Single-variable case}
To simplify the presentation, in this subsection, we assume $x\in \mathbb R$. To steer the solutions of an extremum seeking system towards the direction of high-order Lie brackets, we refer to control approaches from nonholonomic systems theory~\cite{Sus91,Liu97,Mich03,ZuSIAM,Gau14,Gau15,GZ23,GZ24_CDC}. We focus here on the two-input systems of form~\eqref{int_gen}: 
\begin{equation}\label{int_2}
\dot x=g_1(J(x))u_{1}^\varepsilon(t)+g_{2}(J(x)) u_{2}^\varepsilon(t).
\end{equation}
Suppose also  that the dithers $u_1^\varepsilon(t)$, $u_1^\varepsilon(t)$
excite the Lie bracket 
$$
  {\rm ad}_{g_1}^Ng_2(x)=\big[\underbrace{f_1,[f_1,...,[f_1}_{N\text{ times}},f_2]...]\big](z)
$$
at time $t=\varepsilon$,  in the sense that the Chen--Fliess series expansion~\eqref{series_gen}  takes the form
\begin{equation}
    \label{series_2ad}
    x(\varepsilon)=x^0+\varepsilon  {\rm ad}_{g_1}^Ng_2(J(x^0))+R(\varepsilon),
\end{equation}
and all the other Lie brackets of length from 1 to $N$ do not appear in the above expansion. 
One way to construct such inputs is described in~\cite{Gau14,Gau15}, other approaches can be found in, e.g.,~\cite{Sus91,Liu97,PE23}. 
In particular, building on findings of~\cite{Gau14,GZ24_CDC}, we derive the following result on generating the Lie bracket
${\rm ad}_{g_1}^Ng_2$.
\begin{theorem}
\textit{Given system~\eqref{int_2}, let $u_1^\varepsilon(t)=\varepsilon^{-\tfrac{N}{N+1}}c_1v_1^\varepsilon(t)$, $u_2^\varepsilon(t)={\varepsilon^{-\tfrac{N}{N+1}}}c_2v_2^\varepsilon(t)$ with
\begin{equation}
    \label{v_gen_odd}
v_1^\varepsilon(t)= \cos\Big(\frac{2\kappa\pi t}{\varepsilon}\Big), 
v_2^\varepsilon(t)= \sin\left(\frac{2N\kappa\pi t}{\varepsilon}\right),\\
\end{equation}
if $N$ is odd, and 
\begin{equation}
    \label{v_gen_even}
v_1^\varepsilon(t)=\cos\Big(\frac{2\kappa\pi t}{\varepsilon}\Big), 
v_2^\varepsilon(t)=\cos\left(\frac{2N\kappa\pi t}{\varepsilon}\right),\\
\end{equation}
    if $N$ is even. 
Here $\kappa\in\mathbb N$,
\begin{equation}
    \label{cn1cn2}
\begin{aligned}
c_1^Nc_2=(4\pi\kappa)^{N} N!(-1)^{\big[\tfrac{N}{2}\big]}
\end{aligned}
\end{equation}
Then the representation~\eqref{series_gen} for the solutions of system~\eqref{int_2} has the form~\eqref{series_2ad}, i.e.,
$$ x(\varepsilon)=x^0+\varepsilon {\rm ad}_{g_1}^Ng_2(J(x^0))+R(\varepsilon).$$}
\end{theorem}
\begin{proof}
By induction, we have
\begin{equation}
    \label{ad_N}
    {\rm ad}_{g_1}^Ng_2(x)=\sum_{k=0}^N(-1)^k\left(\begin{matrix}
    N\\
    k
\end{matrix}\right)L_{g_1}^{N-k}L_{g_2}L_{g_1}^k(x),
\end{equation}
with $L_{g_1}^{i}L_{g_2}L_{g_1}^j(x):=\underbrace{L_{g_1}... L_{g_1}}_{i\text{ times }}L_{g_2}\underbrace{L_{g_1}... L_{g_1}}_{j-1\text{ times}}g_1(x)$ for $i,j+1\in\mathbb N$, and $L_{g_1}^{i}L_{g_2}L_{g_1}^0(x):=\underbrace{L_{g_1}... L_{g_1}}_{i\text{ times }}L_{g_2}g_1(x)$.
Because of the resonance relation between frequencies of the inputs $v_1$ and $v_2$, all the iterated integrals of length up to  $N$  in the representation~\eqref{series_gen} vanish, which yields
\begin{equation}
    \label{x_ad_N}
     x(\varepsilon)=x^0+ \frac{c_1^Nc_2}{\varepsilon^N}\sum_{k=0}^{N}L_{g_1}^{N-k}L_{g_2}L_{g_1}^k(x^0) I_{k,N+1}(\varepsilon) + R(\varepsilon),
\end{equation}
where
$$
 \begin{aligned}
 I&_{k,N+1}(\varepsilon)=\int\limits_0^\varepsilon\int\limits_0^{s_1}... \int\limits_0^{s_{N+1}}v_{ 1}^\varepsilon(s_1)\cdot...\cdot v_{ 1}^\varepsilon(s_k) \\
 & \cdot v_{ 2}^\varepsilon(s_{k+1}) v_{ 1}^\varepsilon(s_{k+2})\cdot... \cdot v_{ 1}^\varepsilon(s_{N+1})  ds_{N+1}... ds_1.
    \end{aligned}
$$
Direct calculation gives that, for $0\le k\le N$,
\begin{equation}
    \label{I_kN}
    \begin{aligned}
I_{k,N+1}(\varepsilon)=\frac{\varepsilon^{N+1}(-1)^{\big[\tfrac{N}{2}\big]+k }}{(4\pi\kappa)^N k!(N-k)!}.
\end{aligned}
\end{equation}
Thus,
$$
\begin{aligned}
 &c_1^Nc_2 \sum_{k=0}^{N}L_{g_1}^{N-k}L_{g_2}L_{g_1}^k(x^0) I_{k,N+1}(\varepsilon)\\
  &=  \frac{\varepsilon^{N+1}(-1)^{\big[\tfrac{N}{2}\big]}}{(4\pi\kappa)^N }\sum_{k=0}^N\frac{(-1)^{k}}{k!(N-k)!}L_{g_1}^{N-k}L_{g_2}L_{g_1}^k(x^0)\\
 & \overset{\ \eqref{cn1cn2}\ }{\ =\ }\varepsilon^{N+1} (-1)^{2\big[\tfrac{N}{2}\big]}
  \sum_{k=0}^N\frac{(-1)^{k}N!}{k!(N-k)!}L_{g_1}^{N-k}L_{g_2}L_{g_1}^k(x^0)\\
&   \overset{\ \eqref{ad_N}\ }{\ =\ }\varepsilon^{N+1}{\rm ad}_{g_1}^Ng_2(J(x^0)).
\end{aligned}
$$
Substitution of the obtained value in~\eqref{x_ad_N} completes the proof.

\end{proof}

Theorem~1 aligns with the cases considered in~\cite{GZ24_CDC,grushkovskaya2025extremum} for $N=1,2,3$:
    \begin{itemize}
    \item $N=1$: the inputs
    $v_1^\varepsilon(t)=2\sqrt{\kappa\pi}\cos\left({2\kappa\pi t}/{\varepsilon}\right)$, \\
    $v_2^\varepsilon(t)=2\sqrt{{\kappa\pi}}\sin\left({2\kappa\pi t}/{\varepsilon}\right)$, $\kappa\in\mathbb Z$,\\
    excite   $[g_1,g_2]$;
    \item $N=2$: the inputs \\
    $v_1^\varepsilon(t)=\left({4\kappa\pi }\right)^{\frac23}\cos\left({2\kappa\pi t}/{\varepsilon}\right)$,\\
      $v_2^\varepsilon(t)=-2\left({4\kappa\pi }\right)^{\frac23} \cos\left({4\kappa\pi t}/{\varepsilon}\right)$,\\ $\kappa\in\mathbb Z$,
         excite the  Lie bracket $[g_1,[g_1,g_2]]$;
      \item  $N=3$: the inputs \\
      $v_1^\varepsilon(t)=2\left(2\kappa\pi\right)^{\frac34}\sin\left({2\kappa\pi t}/{\varepsilon}\right)$,\\
      $v_2^\varepsilon(t)= -6\left(2\kappa\pi\right)^{\frac34}\cos\left({6\kappa\pi t}/{\varepsilon}\right)$, $\kappa\in\mathbb Z$,
         excite the  Lie bracket $\big[g_1,[g_1,[g_1,[g_1,g_2]]]\big]$.
\end{itemize}
Some further choices will be considered in Section~IV.
\begin{remark}[about relation to~\cite{PE23}]
    By examining \eqref{eqn:LBS_main} in light of the design choices above, we guarantee that we achieve a complete average asymptote for the LBS as in \cite[Section 4.1]{PE23} (i.e., higher-order LBS approximations are made redundant by that design in that they vanish as $\varepsilon \rightarrow 0$). That is, for $N=1$ our design choice of $v_1$ and $v_2$ to excite $[g_1,g_2]$ guarantees that $p_1+p_2=1/2+1/2=1$, which guarantees bounded first-order LBS as $\varepsilon \rightarrow 0$. Same hold for $N=2$ as our design choices of $v_1$ and  $v_2$ to excite $[g_1,[g_1,g_2]]$ guarantees that $p_1+p_1+p_2=2/3+2/3+2/3=2$. Due to non-resonance conditions, coefficients resulting from the iterated integrals for first-order terms and other second-order terms (except the excited bracket) vanish. This also leads to a bounded second-order LBS as $\varepsilon \rightarrow 0$.
    Lastly, the design choices of $v_1$ and  $v_2$ to excite $\big[g_1,[g_1,[g_1,g_2]]]\big]$ guarantees that $p_1+p_1+p_1+p_2=3/4+3/4+3/4+3/4=3$. Due to non-resonance condition, coefficients resulting from the iterated integrals for first-order terms, second-order terms and other third-order terms (except the excited bracket) vanish. This also leads to a bounded third-order LBS as $\varepsilon \rightarrow 0$.
\end{remark}

Assume further that $g_1,g_2$ are chosen in such a way that the Lie bracket $ {\rm ad}_{g_1}^Ng_2(J(x))$ has the form
\begin{equation}\label{g_choice}
 {\rm ad}_{g_1}^Ng_2(J(x))=-\sigma J^{(N)}(x),
\end{equation}
 and some $\sigma >0$ playing a role of control gain parameter. 
In this context, $ {\rm ad}_{g_1}^Ng_2(J(x))$ denotes the corresponding Lie bracket computed with the compositions of functions $g_1\circ J$, $g_2\circ J$.
The most obvious choice of the vector fields satisfying relation~\eqref{g_choice} is $g_1(z)\equiv 1$, $g_2(z)=-z$. For the first-order Lie brackets, the whole family of functions $g_1,g_2$ satisfying the relation
\begin{equation}
    \label{g_choice_2}
    [g_1\circ J,g_2\circ J](x)=-\varphi(J(x))\nabla J(x),
\end{equation}
with any given continuous  function $\varphi$,
has been introduced in~\cite{GZE18}.  Then the expansion~\eqref{series_2ad} takes the form
\begin{equation}
    \label{series_2eps}
    x(\varepsilon)=x^0-\varepsilon \sigma J^{(N)}(x^0)+R(\varepsilon), 
\end{equation}
and, similarly to the approach of~\cite{GZE18}, the practical asymptotic stability of $x^*$ for system~\eqref{int_2} can be proved. We proceed by summarizing the key results of this subsection and integrating them into  the context of solving the extremum seeking problem. For this purpose, we further  specify the properties of  $J$ as a polynomial-like single-variable function:
\begin{assumption}
    The function $J\in C^m(D,\mathbb R^{m})$ \textcolor{black}{satisfies assumption A1.1)}, $D\subset \mathbb R$, with some $m\ge 2$, and
there exist constants  $\alpha_{1},\alpha_{2},
\beta_{1},\beta_{2}$, such that, for all $x\in D$,
     $$ 
     \begin{aligned}
   \alpha_{1}\|x-x^*\|^{m} \le  J(x)-J^* & \le \alpha_{2}\|x-x^* \|^{m},\\
    J^{(m-1)}(x)(x-x^*)&\ge\beta_{1}\|x-x^*\|^2,\\
  \|J^{(m-1)}(x)\|&\le \beta_{2}\|x-x^*\|.
    \end{aligned}
    $$
\end{assumption}
\begin{remark}
    \textcolor{black}{Assumption 2 means that $J$ is convex and has an isolated minimum in the domain $D$ at $x^*$ attained at $J^*$. Similar to A1.2), it reflects the need for $J$ to behave locally like a polynomial of power $m$ but with clear emphasis on the local behavior of the $(m-1)$ derivative of $J$, $J^{(m-1)}$ representing a linear-like term. For example, if $J=x^6$, then $m=6$, $x^*=J^*=0$ and $J^{(m-1)}=(6!)(x)$. Clearly $(6!)(x)$ is bounded above by a linear function and if multiplied by a linear term, it can be bounded below by a quadratic function.      
    }
\end{remark}
\begin{figure}[t]
    \centering
    \includegraphics[width=1\linewidth]{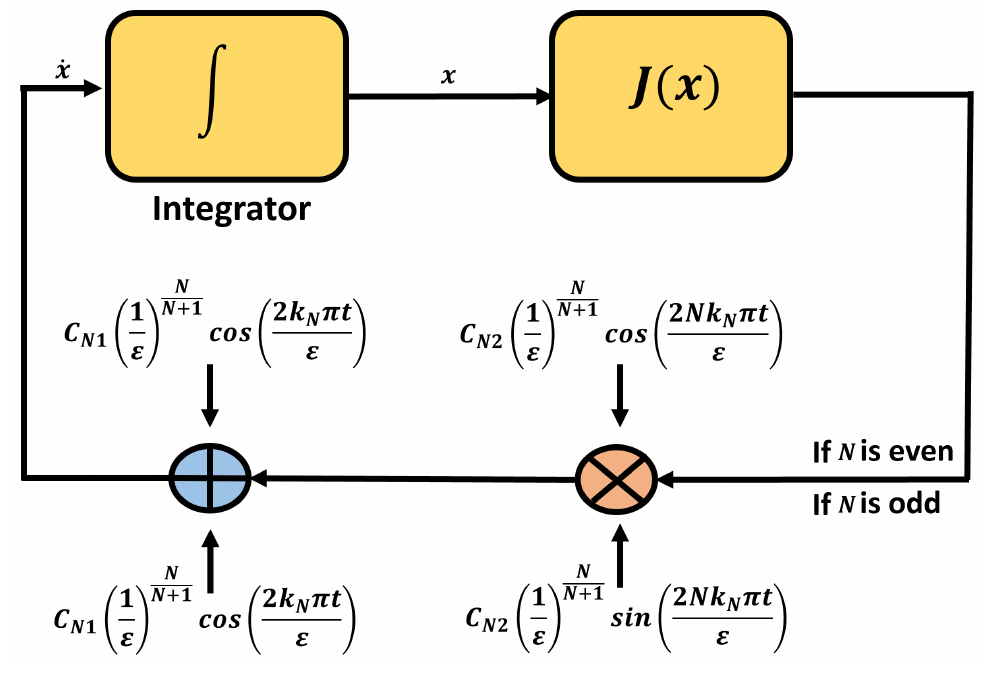}
    \caption{A schematic representation for the generalized two-input control design proposed in Theorem 1. For an odd $N$ the two input excitation signals follow \eqref{v_gen_odd} and for an even $N$ the two input excitation signals follow \eqref{v_gen_even}. The coefficients $c_1$ and $c_2$ can be chosen as desired as long as the condition in \eqref{cn1cn2} is satisfied.}
    \label{block_diagram}
\end{figure}
Under the above assumption, the control inputs introduced in Theorem~1 can be used for designing extremum seeking algorithms with exponential convergence to an arbitrary small neighborhood of $x^*$, as it is stated in the following theorem (see also Fig.~\ref{block_diagram} for a schematic representation).
\begin{theorem}\label{thm_2inp}
\textit{Given system~\eqref{int_2}  and a cost function $J$ satisfying Assumption~2, let the vector fields $g_1,g_2\in C^m(D;\mathbb R^n)$ satisfy the relation~\eqref{g_choice}, and let $u_1^\varepsilon,u_2^\varepsilon$ be defined as in Theorem~1 with  $N=m-1$. Then the point $x^*$ is semi-globally practically exponentially stable for system~\eqref{int_2}.
Namely, for any $\delta$ such that $\overline{B_\delta(x^*)}\subset D$,  any $\bar\gamma\in(0,\sigma \beta_1)$, and $\rho\in(0,\delta)$, there exists an $\bar\varepsilon>0$ such that, for any $\varepsilon\in(0,\bar\varepsilon]$,
$\gamma\in(0,\bar\gamma]$, the solutions of system~\eqref{int} with $x^0{\in} B_{\delta}(x^*)$ exhibit the following decay rate:}\\
    $\|x(t)-x^*\|\le \lambda \|x^0-x^*\|e^{-\gamma t}+\rho\text{ for al }t\ge0.$
\end{theorem}
\begin{proof}
The argumentation follows a similar line as the proof of practical exponential stability in \cite[Theorem 3]{GZE18}, so we only briefly highlight here the main steps that differ from that proof.
Given any $\delta>0$, let $D'$ be any compact set such that $\overline{B_\delta(x^*)}\subset D'\subset D$ . From the integral representation  of the solutions of  system~\eqref{int_2} with $x(0)=x^0\in\overline{B_\delta(x^*)}$ and $u^\varepsilon$ defined in the theorem, we obtain 
\begin{equation}
    \label{x_Est}
    \|x(t)-x^0\| \le  c_x\varepsilon^{\frac{1}{m+1}},
\end{equation}
 where $c_x= \sup\limits_{x\in D'}(c_1\|g_{k_1}\circ J(x)\|+c_2\|g_{k_1}\circ J(x)\|)$. Then taking $\varepsilon_0=(\frac{{\rm dist}(\delta D',\delta D)-\delta}{c_x})^{m+1}$ we ensure that the solutions of system~\eqref{int_2} with $x^0\in B_\delta(x^*)$ and $\varepsilon\in(0,\varepsilon_0)$ are well-defined in $D'$ for all $t\in[0,\varepsilon]$.
 
  To estimate the remainder in~\eqref{series_2eps}, denote \\
  $$
  \begin{aligned}
  c_{R}=\sup\limits_{x\in D'}\sum\limits_{k_1,..., k_{m+1}=1}^{2}&c_{k_1}\cdot...\cdot c_{k_{m+1}}\\
  &\cdot\|L_{g_{k_{m+1}}}... L_{g_{k_2}}g_{k_1}\circ J(x)\|,
  \end{aligned}
  $$ 
  Then it is easy to see that the remainder in~\eqref{series_gen} can be estimated as
   $$ \begin{aligned}
\|R(\varepsilon)\|&\le\varepsilon^{-m}\hspace{-1em}\sum_{k_1,..., k_{m+1}=1}^{2}\max_{0\le t\le\varepsilon}|u_{k_1}^\varepsilon(s_1)\cdot ...  \cdot u_{k_{m+1}}^\varepsilon(s_{m+1})|
\\
\cdot\int\limits_0^\varepsilon\int\limits_0^{s_1}&... \int\limits_0^{s_{m}}\|L_{g_{k_{m+1}}}... L_{g_{k_2}}g_{k_1}\circ J(x(s_{m+1}))  \| ds_{m+1}... ds_1\\
       &\le  c_R\varepsilon^{1+\frac1m}.
    \end{aligned}
  $$
The above estimate together with the representation~\eqref{series_2eps}, Assumption~2, and triangular inequality, implies that, for any $\varepsilon\in(0,\varepsilon_0)$,
$$
\begin{aligned}
& \|x(\varepsilon)-x^*\|^2=\|x^0-x^*-\varepsilon\sigma J^{(m-1)}(x^0)+R(\varepsilon)\|^2\\
&\le \|x^0-x^*\|^2(1-2\varepsilon\bar\gamma)+\varepsilon^{1+\tfrac{1}{m}}(c_{x1}\|x^0-x^*\|+c_{x2}),\\
\end{aligned}
$$
with $\bar\gamma=\beta_1\sigma $, $c_{x1}=2c_R(1+\varepsilon\sigma \beta_2)$, $c_{x2}=\varepsilon_0^{1-\frac{1}{m}}(\sigma ^2\beta_2^2+ \varepsilon_0^{\tfrac{2}{m}}c_R^2)$.
Given any $\rho\in(0,\delta)$ and any $\gamma\in(0,\bar\gamma)$,  let $\hat\rho\in(0,\rho)$, $\varepsilon_1=\left(\frac{2(\bar\gamma-\gamma)\hat\rho^2}{c_{x1}\hat\rho+c_{x2}}\right)^m$. Given any  $\varepsilon\in(0,\varepsilon_1)$,
consider two cases:\\
\emph{Case 1.} $\|x^0-x^*\|<\hat\rho$. Then 
$$
\begin{aligned}
  \|x(\varepsilon)-x^*\|^2&\le \hat\rho^2-\varepsilon(2\bar\gamma\hat\rho^2-\varepsilon^{\frac{1}{m}}(c_{x1}\hat\rho+c_{x2}))\\
  &\le \hat\rho^2(1-2\varepsilon\gamma)<\hat\rho<\delta.
\end{aligned}
$$
Since $x(\varepsilon)\in\overline{B_\delta(x^*)}$, we can repeat all the previous reasoning for the time interval $[\varepsilon,2\varepsilon]$ with the same parameters as before to ensure that the solutions of system~\eqref{int_2} are well-defined in $D'$ for all $t\in[\varepsilon,2\varepsilon]$.
Furthermore, as estimate~\eqref{x_Est}  holds for any $\varepsilon$-length time interval (whenever the solutions are well-defined in $D'$),
 we can take $\varepsilon_2=\big(\frac{\rho-\hat\rho}{c_x}\big)^{m+1}$ to ensure that $\|x(\varepsilon)-x(t)\|\le \rho-\hat\rho $ (and $<\rho$) for all $t\in[\varepsilon,2\varepsilon]$ provided that $\varepsilon\in(0,\varepsilon_2)$,
 and therefore,
$$
\begin{aligned}
    \|x(t)-x^*\| &\le \|x(\varepsilon)-x^*\|+\|x(\varepsilon)-x(t)\|\\
    &\le \hat\rho+\rho-\hat\rho=\rho \text{ for all }t\in[\varepsilon,2\varepsilon].
\end{aligned}
$$
 
\emph{Case 2.} $\|x^0-x^*\|\ge \hat\rho$. Then
$$
\|x(\varepsilon)-x^*\|^2\le \|x^0-x^*\|^2(1-2\varepsilon\gamma)\le \|x^0-x^*\|^2e^{-2\varepsilon\gamma}.
$$
Again,  $x(\varepsilon)\in\overline{B_\delta(x^*)}$, which implies that the solutions of system~\eqref{int_2} are well-defined in $D'$ for all $t\in[\varepsilon,2\varepsilon]$, and, in case $\|x(\varepsilon)-x^*\|\ge \hat\rho$, we get the estimate
$$
\|x(2\varepsilon)-x^*\|\le  \|x(\varepsilon)-x^*\|e^{-\varepsilon\gamma}\le \|x^0-x^*\|e^{-2\varepsilon\gamma},
$$
and, moreover, for all $t\in[\varepsilon,2\varepsilon]$,
$$
\begin{aligned}
    \|x(t)-x^*\| &\le \|x(\varepsilon)-x^*\|+\|x(\varepsilon)-x(t)\|\\
    &\le \|x^0-x^*\|e^{-\varepsilon\gamma}+\rho \text{ for all }t\in[\varepsilon,2\varepsilon].
\end{aligned}
$$
Repeating considerations from case 1) and case 2), we come to the following conclusion:  
$$
\begin{aligned}
& \|x(j\varepsilon)-x^*\|\le \|x^0-x^*\|e^{-j\varepsilon\gamma}\text{ for all }j=0,\dots,\mathcal N-1,\\
&\|x(t)-x^*\|\le \rho\text{ for all }t\ge \mathcal N\varepsilon,
\end{aligned}
$$
where  $\mathcal N $ is the smallest non-negative integer  number such that $\|x^0-x^*\|e^{-\mathcal N\varepsilon\gamma}<\rho$. 
For arbitrary $t\in[0,\infty)$,  denote the integer part of $t/\varepsilon$  as $t_{in}^\varepsilon$
  and note  that  $\|x(t)-x(t_{in}^\varepsilon\varepsilon)\|\le \rho$ for all $t\ge 0$ because $0\le t-t_{in}^\varepsilon\le \varepsilon$. Then
  $$
  \begin{aligned}
    \|x(t)-x^*\|&\le \|x(t_{in}^\varepsilon\varepsilon)-x^*\|+\|x(t)-x(t_{in}^\varepsilon\varepsilon)\|\\
    &\le \|x^0-x^*\|e^{-t_{in}^\varepsilon \varepsilon\gamma}+\rho\\
    &\le \lambda\|x^0-x^*\|e^{- \gamma t}+\rho\text{ for all }t\ge 0,
     \end{aligned}
  $$
with $\lambda=e^{\gamma\varepsilon}$.

\end{proof}

\begin{remark}[about regularity assumptions]
As in the paper~\cite{GZE18}, it is also possible to relax regularity assumption on the control vector fields. Namely, requirement of  $g_k\circ J$, $k=1,2$, being $m$ times continuously differentiable in $D$ can be replaced with the following: $g_k\circ J\in C^m(D\setminus\{x^*\});\mathbb R$ and  $L_{g_{k_{N+1}}}... L_{g_{k_2}}g_{k_1}\circ J\in C(D;\mathbb R)$ for all $N\in\{1,...,m\}$, $k_1,...,k_{m+1}\in\{1,2\}$. This relaxation is particularly important for deriving conditions for the ``classical'' exponential stability in the sense of Lyapunov, meaning that the trajectories of the extremum seeking system converge to the point
 $x^*$, rather than merely to its neighborhood. Another important condition for achieving classical exponential stability is the property of vanishing amplitudes, which requires that $g_k\circ J\to 0$  as $x\to x^*$ (see~\cite[Theorem~3, Part II]{GZE18}). Since selecting  such vector fields becomes increasingly challenging in the case of higher-order Lie brackets, we leave this task, along with a rigorous formulation of the corresponding exponential stability properties, for furture research.
\end{remark}
\begin{remark}[about decay rate]
Under the conditions of the theorem and the assumption $N = m-1$, the solutions of system~\eqref{int_2} converge to a neighborhood of $x^*$ with an exponential decay rate $O(e^{-\gamma t})$. Under additional assumptions, it is also possible to establish practical asymptotic stability with a polynomial decay rate in the case $N < m-1$, similarly to Lemma~2. For example, in case $J(x)=(x-x^*)^m$ and control inputs from Theorem~1 with $1\le N<m-1$, one can obtain the estimate 
 $$
    \|x(t)-x^*\|\le \Big(\lambda_{m}\|x^0-x^*\|^{1+N-m}+\tilde  \lambda_m t\Big)^{-1/(m-N-1)}+\rho
    $$
    with some $\lambda_m,\tilde\lambda_m,\rho>0$,
using argumentation similar to the proof of~\cite[Theorem~3]{GZE18}.
\end{remark}
\subsection{Multi-variable case}
The approach introduced in the previous subsection can be extended to the case $x \in \mathbb{R}^n$. In this subsection, we consider a special case and leave more general conditions for a separate study. In particular, we introduce the following assumptions. 
\begin{assumption}
    The function $J\in C^m(D,\mathbb R^{m})$ \textcolor{black}{satisfies A1.1)}, $D\subset \mathbb R^n$, with some $m\ge 2$, and
there exist constants  $ \alpha_1,\alpha_2,\beta_{1},\beta_{2}$, such that, for all $x\in D$,
     $$ 
     \begin{aligned}
    \alpha_{1}\|x-x^*\|^{m} \le  J(x)-J^* & \le \alpha_{2}\|x-x^* \|^{m},\\
 \sum_{i=1}^n  \frac{\partial^{m-1}J(x)}{\partial x_i^{m-1}} (x_i-x_i^*)&\ge\beta_{1}\|x-x^*\|^2,\\
  \sum_{i=1}^n \left| \frac{\partial^{m-1}J(x)}{\partial x_i^{m-1}}\right|&\le \beta_{2}\|x-x^*\|.
    \end{aligned}
    $$
\end{assumption}
In particular, the above assumption holds if $J=\sum_{i=1}^nJ_i(x_i)$ with $J_i$ satisfying Assumption~1 for all $i=1,2,\dots,n$.

If the Assumption~3 is satisfied, then, analogously  to the previous  subsection, we  design an extremum seeking system
\begin{equation}\label{int_2n}
\dot x_i=g_{1i}(J(x))u_{1i}^\varepsilon(t)+g_{2i}(J(x)) u_{2i}^\varepsilon(t),\, i=1,\dots,n,
\end{equation}
where for each $i$, the pairs of dithers $u_{1i}^\varepsilon(t)$, $u_{2i}^\varepsilon$ are defined as in Theorem~1 to excite the Lie bracket ${\rm ad}_{g_{1i}}^{m-1}g_{2i}$. In order to ensure that no additional Lie brackets of the same or shorter  length are generated, we impose a non-resonance assumption on the control frequencies, similarly to~\cite[Assumption~1]{GZ18},\cite[Assumption~2]{GZ24_CDC}.
\begin{assumption}
The numbers $\kappa_1,\kappa_2,\dots,\kappa_n\in\mathbb N$ satisfy the following property: for any set of numbers $\mu_1,\dots,\mu_m,\nu_1,\dots,\nu_n\in\{0,\pm,1,\dots,\pm N\}$ such that $\sum\limits_{i=1}^n|\mu_i|+|\nu_i|\le N$,
the property $\displaystyle \sum_{i=1}^n(\mu_i+N\nu_i)\kappa_i=0$ implies $\mu_i=-N\nu_i$ for all $i$.
\end{assumption}

\begin{theorem}
    \label{thm_2n}
\textit{Given system~\eqref{int_2n}
and a cost function $J$ satisfying Assumption~3, let the vector fields $g_{1i},g_{2i}\in C^m(D;\mathbb R^n)$ pair-wisely satisfy the relation~\eqref{g_choice}
with some $\sigma>0$,
and let  $u_{1i}^\varepsilon(t)=\varepsilon^{-\tfrac{m-1}{m}}c_1v_{1i}^\varepsilon(t)$, $u_{2i}^\varepsilon(t)={\varepsilon^{-\tfrac{m-1}{m}}}c_2v_{2i}^\varepsilon(t)$ with
$$v_{1i}^\varepsilon(t)= \cos\Big(\frac{2\kappa_{i}\pi t}{\varepsilon}\Big), 
v_{2i}^\varepsilon(t)= \sin\left(\frac{2(m-1)\kappa_{i}\pi t}{\varepsilon}\right),\\
$$
if $m-1$ is odd, and 
$$v_{1i}^\varepsilon(t)=\cos\Big(\frac{2\kappa_{i}\pi t}{\varepsilon}\Big), 
v_{2i}^\varepsilon(t)=\cos\left(\frac{2(m-1)\kappa_{i}\pi t}{\varepsilon}\right),\\
$$
    if $i$ is even. 
Here $c_1$, $c_2$ satisfying~\eqref{cn1cn2} with $N=m-1$, and $\kappa_{i}\in\mathbb N$ satisfy the non-resonance Assumption~4 with $N=m-1$.
  Then the point $x^*$ is semi-globally practically exponentially stable for system~\eqref{int_2}.
}
\end{theorem}
\begin{proof}
    The proof is the same as for Theorem~2 taking into account that the expansion~\eqref{series_gen} takes the form
    $$
    \begin{aligned}
  x(\varepsilon)=&x^0+\varepsilon \sum_{i=1}^n {\rm ad}_{g_{1i}}^Ng_{2i}(J(x^0))e_i+R(\varepsilon)\\
  =&x^0-\varepsilon \sigma \sum_{i=1}^n  \frac{\partial^{m-1}J(x^0)}{\partial x_i^{m-1}}  e_i+R(\varepsilon)
    \end{aligned}
    $$
with $e_i$ denoting the $n$-dimensional unit vector with non-zero $i$-th entry.  Then Assumption~3 yields the exponential decay rate estimate. 
\end{proof}

\begin{figure*}[ht]
    \centering
        \centering        \includegraphics[width=0.9\linewidth]{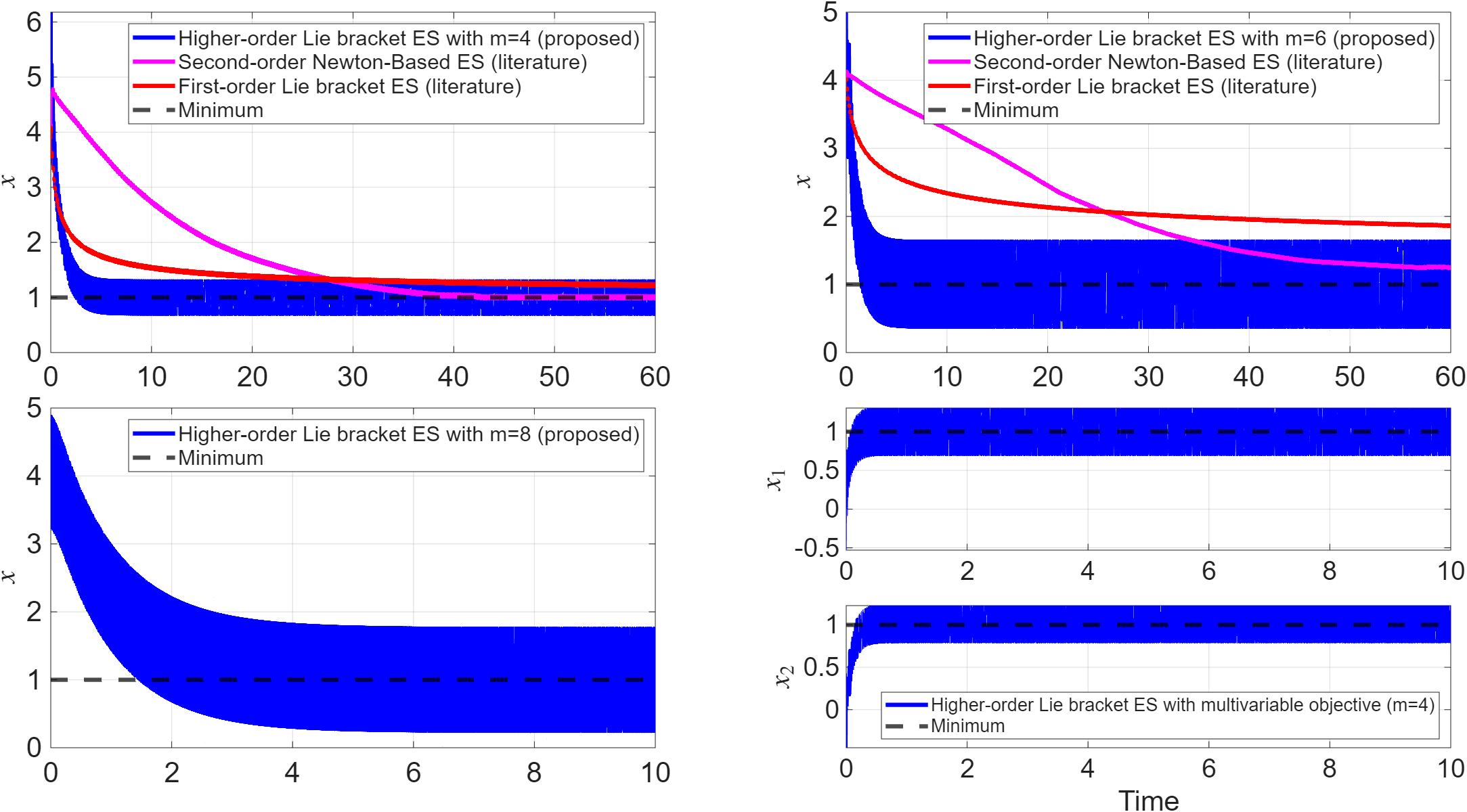}
        \caption{\textcolor{black}{
        Time plots of the proposed ES minimizing $J=\frac{1}{m!}x^m$ for $m=4$ with $\varepsilon=10^{-3}$ (top-left) and $m=6$ with $\varepsilon=10^{-4}$ (top-right) along with comparisons with gradient-based and Newton-based ES methods from literature, demonstrating the exponential convergence advantage of the proposed ES. For $m=8$ with $\varepsilon=10^{-6}$, time plot (bottom-left) shows exponential convergence of the proposed ES. Additionally, with $m=4$ the multi-variable objective function $J(x)=(x_1-1)^4+(x_1-x_2)^4$ is minimized using the proposed muti-variable ES with $\varepsilon=10^{-3}$; $x_1(t)$ and $x_2(t)$ both converges exponentially to the minimum.}
}
    \label{fig_pol4}
\end{figure*}
\section{Numerical simulations}
\subsection{Single-variable cost function}\label{sec:1D_Numerical}
 To demonstrate the effectiveness of the proposed approach, 
we take a cost function as the $m$-th order polynomial:
$$J_m(x)=\frac{1}{m!}(x-1)^m,
$$
and $g_1(z)\equiv 1$, $g_2(z)=-z$, 
so that $ {\rm ad}_{g_1}^{m-1}g_2(J(x))=-\sigma J^{(m-1)}(x)$.
We  apply first-order-based approach~\cite{DurrAuto},
\begin{equation}
    \label{ex1_Durr}
\dot x=2\sqrt{\frac\pi\varepsilon}\left(J_m(x)\cos\frac{2\pi t}{\varepsilon}+\sin\frac{2\pi t}{\varepsilon}\right),
\end{equation}
and higher-order-based approach with two inputs defined as in Theorem~1:
 \begin{equation}
    \label{ex1_m}
\dot x= C_m\left(\cos\frac{2\pi t}{\varepsilon}+(-1)^{\frac{m}{2}}J_m(x)\sin\frac{2(m-1)\pi t}{\varepsilon}\right),
\end{equation}
where $C_m= \left((m-1)!{\Big(\dfrac{4\pi}{\varepsilon}\Big)^{m-1}}\right)^{1/m}$. \textcolor{black}{Additionally, we include comparisons with the Newton-based ES approach from \cite{Labar19} in similar fashion and details to \cite[Examples 2 and 3]{PE23} for the cases of $m=4$ and $m=6$; we could not include a comparison with $m=8$ as numerical solvers struggle severely with the Newton-based approach given the extreme small values of the Hessian near the minimum.}

\textcolor{black}{In all  cases, we  initiate the  equations at \textcolor{black}{$x(0)=4$}. We use $\varepsilon=10^{-3}$ for $m=4$, $\varepsilon=10^{-4}$ for $m=6$, and $\varepsilon=10^{-6}$ for $m=8$.} 
\textcolor{black}{Figure \ref{fig_pol4} shows successful exponential convergence of the proposed ES for $m=4$ (top-left), $m=6$ (top-right) and $m=8$ (bottom-left). Moreover, for $m=4$ and $m=6$ it is clear that the proposed ES outperformed the gradient-based and Newton-based methods substantially. The proposed ES achieved successful convergence within few seconds for $m=4,6,8$ as opposed to gradient-based and Newton-based methods unsuccessful convergence for $m=6$ after 60 seconds; Newton-based method took over 40 seconds to converge in the case of $m=4$ while gradient-based method did not converge even after 60 seconds.} 
\subsection{Multi-variable cost function}
In this subsection, we illustrate Theorem~3 with $x\in\mathbb R^2$, 
$$
J(x)=(x_1-1)^4+(x_1-x_2)^4.
$$
It can be seen that the assumptions of Theorem~3 are satisfied, in particular, $x^*=(1,1)^\top$,
$$ 
     \begin{aligned}
    \alpha_{1}\|x-x^*\|^{4} \le  J(x)-J^* & \le \alpha_{2}\|x-x^* \|^{4},\\
 \frac{\partial^{3}J(x)}{\partial x_1^{3}} (x_1-x_1^*)+\frac{\partial^{3}J(x)}{\partial x_2^{3}} (x_2-x_2^*)&=\\
24(x_1-x_1^*)^2+24(x_1-x_2)^2 &\ge \beta_{1}\|x-x^*\|^2,\\
\left| \frac{\partial^{3}J(x)}{\partial x_1^{3}}\right| +\left|\frac{\partial^{3}J(x)}{\partial x_2^{3}}\right| &=\\
 24(|x_1-x_1^*|+2|x_1-x_2|) &\le \beta_{2}\|x-x^*\|.
    \end{aligned}
    $$
with $\alpha_1=\frac{1}{18}$, $\alpha_2=5$, $\beta_1=8$, $\beta_2=24\sqrt{13}$. According to the approach of Section~III.B, we define the extremum seeking system as
\begin{equation}
    \label{ex_2D_we}
    \begin{aligned}
&\dot x_i=C_{i}\left(\cos\frac{2\pi\kappa_i t}{\varepsilon}+J(x)\sin\frac{6\pi\kappa_i t}{\varepsilon}\right), 
    \end{aligned}
\end{equation}
where $i=1,2$, $C_i=6^{1/4} \Big(\dfrac{4\pi\kappa_i}{\varepsilon}\Big)^{3/4}$, and $\kappa_1,\kappa_2$ satisfy Assumption~4. 
For numerical simulations, we put $\kappa_1=1$, $\kappa_2=4$, $\varepsilon=10^{-3}$, $x(0)=(0,0)^\top$. \textcolor{black}{Figure \ref{fig_pol4} (bottom-right)} shows the time plots of \textcolor{black}{$x_1(t)$ and $x_2(t)$, demonstrating exponential convergence to the the minimum}. 
\subsection{Limitations in Testing Semi-Global Convergence}\label{sec:limitation}
\textcolor{black}{While the theoretical results of the proposed ES approach can attain semi-global convergence for functions that satisfy Assumption 2 semi-globally (i.e., $D$ can be made as large as needed) such as $J_m(x)$ in Section \ref{sec:1D_Numerical}, numerical simulations of such functions can be very challenging for initial conditions further away from the minimum due to the extremely large values of $J_m(x)$. Figure \ref{fig:limitation} includes a simulation for $J_m(x)$ for $m=4$, $\varepsilon=10^{-4}$ and $x(0)=\pm 5$, resulting in $\dot{x}(0)\approx 10^4$, which is extremely large. This motivates the need for bounded update ES laws similar to \cite{Sch14} but with the proposed ES exponential convergence property -- see similar observations from practice in \cite[Sections 4.7-4.9]{palanikumar2026model}.}
\begin{figure}[ht]
    \centering    \includegraphics[width=0.9\linewidth]{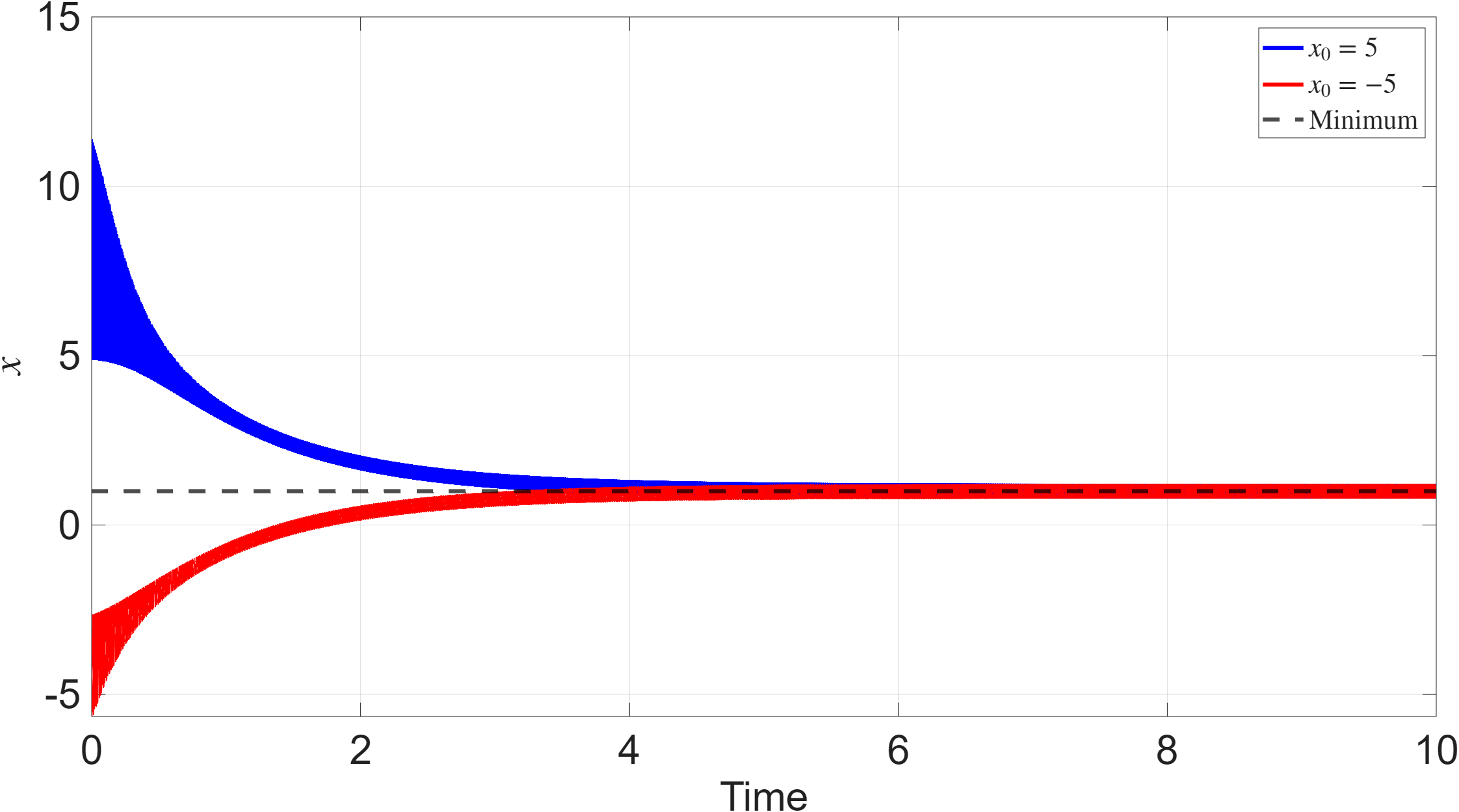}
    \caption{Simulation of $J_m$ in Section \ref{sec:1D_Numerical} for $m=4$, $\varepsilon=10^{-4}$ and $x(0)=\pm 5$, resulting in $\dot{x}(0)\approx 10^4$, which is extremely large and pose some challenges in numerical simulations/implementations as mentioned in Section \ref{sec:limitation}.}
    \label{fig:limitation}
\end{figure}

\section{Conclusions \& Future Work}
In this paper, we have introduced a novel framework for \textcolor{black}{optimization and control} via extremum seeking that utilizes higher-order Lie bracket approximations to achieve improved convergence properties. Unlike many ES approaches that rely on first-order Lie bracket systems and yield exponential convergence only for quadratic-like cost functions, the proposed approach may ensure the exponential convergence even when the cost function is (\textcolor{black}{``flat-bottomed"}), i.e., behaves locally like $\|x - x^*\|^m$ with $m > 2$. This is achieved by integrating ideas from differential geometric control theory related to higher-order Lie bracket averaging for control-affine systems~\cite{PE23} with control design techniques for high-order nonholonomic systems~\cite{GZ24_CDC,Gau14}, and advanced analytical tools for studying dynamical systems, in particular, the Chen--Fliess series expansion \cite{GZE18,GZ23}.

 In future work, we plan to extend this approach to multi-variable cost functions under less restrictive assumptions on their local behavior, as well as to \textcolor{black}{extremum seeking with general dynamic systems} as in \cite{GE21}. We also intend to explore potential applications and implementations of the proposed results in related control problems. In particular, in our recent work~\cite{palanikumar2026model}, we applied this approach to a unicycle-type system and validated it experimentally, \textcolor{black}{which was achieved in a discrete setting; hence, a promising future direction will be expanding the proposed ES approach to discrete framework.} Furthermore, the explicit control functions proposed in Theorem~1 can be used in control and motion planning of nonholonomic systems. Another important research direction concerns possible choices of generating vector fields in Lemma~3. 
 We expect that, as in~\cite{GZE18}, it is possible to address the oscillation problem by appropriately choosing vector fields that vanish at the extremum. 

All in all, we believe that this paper opens new possibilities for designing extremum seeking optimization and control laws and  initiates a promising direction for further developments of high-order Lie bracket methods in optimization tasks. 



\section*{References}
\bibliographystyle{ieeetr}
\bibliography{biblio_ES}
\end{document}